\documentclass[12pt]{amsart}
\usepackage{latexsym, amsmath, amssymb,amsthm,amsopn,amsfonts}
\usepackage{version}
\usepackage{epsfig,graphics,color,graphicx,graphpap,dsfont}
\usepackage{amssymb}

\ifx\pdfoutput\undefined
  \DeclareGraphicsExtensions{.eps}
\else
  \ifx\pdfoutput\relax
    \DeclareGraphicsExtensions{.eps}
  \else
    \ifnum\pdfoutput>0
      \DeclareGraphicsExtensions{.pdf}
    \else
      \DeclareGraphicsExtensions{.eps}
    \fi
  \fi
\fi

\newcommand{\LV}{\left|}
\newcommand{\RV}{\right|}
\newcommand{\LN}{\left\|}
\newcommand{\RN}{\right\|}
\newcommand{\LB}{\left[}
\newcommand{\RB}{\right]}
\newcommand{\LC}{\left(}
\newcommand{\RC}{\right)}
\newcommand{\LA}{\left<}
\newcommand{\RA}{\right>}

\theoremstyle{plain}

\newtheorem{thm}{Theorem}[section]
\newtheorem{prop}{Proposition}[section]
\newtheorem{cor}[prop]{Corollary}
\newtheorem{lem}[prop]{Lemma}

\theoremstyle{definition}

\newtheorem{rem}{Remark}[section]

\numberwithin{equation}{section}

\def\squarebox#1{\hbox to #1{\hfill\vbox to #1{\vfill}}}

\usepackage{amsxtra}

\title[On the well-posedness of a weakly dispersive Boussinesq system]{On the well-posedness of a weakly dispersive one-dimensional Boussinesq system}

\author[R. M. Chen]
{Robin Ming Chen}
\address{Robin Ming Chen\newline
Department of Mathematics\\
University of Pittsburgh\\
Pittsburgh, PA 15260} \email{mingchen@pitt.edu}
\author[Y. Liu]
{Yue~Liu}
\address{Yue Liu\newline
University of Texas at Arlington, Department of Mathematics, Arlington, TX 76019-0408}
\email{yliu@uta.edu}

\thanks{The work of R. M. Chen was partially supported by the NSF grant DMS-0908663. The work of Y. Liu was partially supported by the NSF grant DMS-0906099 and the NHARP grant 003599-0001-2009.}

\begin{document}

\begin{abstract}
We study the Cauchy  problem for one-dimensional dispersive system of Boussinesq type which models weakly nonlinear long wave surface waves. We establish the local well-posedness and ill-posedness of solutions to the system. We also provide criteria for the formation of singularities.
\end{abstract}
\maketitle






%

\section{Introduction}\label{sec_intro}


The theory of water waves embodies the equations of fluid mechanics,
the concepts of wave propagation, and the critically important role
of boundary dynamics. Due to the complexity of the governing equations for water waves, physicists and mathematicians are led to derive simpler sets of equations likely to describe the dynamics of the water waves in some specific physical regimes.

One regime that arises in practical situations is that of waves in a channel of approximately
constant depth $h$ that are uniform across the channel, and which are of small
amplitude and long wavelength. Similar waves appear as long-crested disturbances on larger bodies of water. Let $a$ be a typical wave amplitude and $\lambda$ a typical wavelength, the conditions above amount to
\begin{equation*}
\varepsilon = {a\over h} \ll1, \quad \delta = {h^2\over \lambda^2} \ll 1.
\end{equation*}

The equations within the above scaling regime couple the free surface elevation $\eta$ to the horizontal component of the velocity $v$. When the nonlinear and dispersive effects are balanced, that is, $\varepsilon = O(\delta)$, and for one dimensional surfaces and flat bottoms, at the asymptotic expansion to the first order in $\varepsilon$, the system reduces to the well-known Boussinesq equation \cite{Bous, Craig, KanoNishida, Ursell}
\begin{equation*}
v_{tt} - v_{xx} - {3\varepsilon\over2} (v^2)_{xx} - {\varepsilon\over3} v_{xxxx} = 0,
\end{equation*}
(note that $\eta = v + O(\varepsilon)$) and any perturbation of the surface splits up into two components moving in opposite directions.

Taking advantage of the freedom associated with the choice of the velocity
variable and making full use of the lower-order relations (the wave equation written
as a coupled system) in the dispersive terms, Bona-Chen-Saut [9] put forward the following three-parameter family of Boussinesq systems (referred to as the $abcd$ system) for one dimensional surface (a two-dimensional analogue is derived by Bona-Colin-Lannes \cite{BonaColinLannes}),
\begin{equation}\label{abcdsys}
\left\{
\begin{split}
 & \eta_t + v_x + (\eta v)_x + a v_{xxx} - b \eta_{txx} = 0 \\
 & v_t + \eta_x + vv_x + c \eta_{xxx} - d v_{txx} = 0,
\end{split}\right.
\end{equation}
which is formally equivalent models of solutions of the Euler equations. In the above system, $\eta(t,x)$ is proportional to the deviation of the free surface from its rest position, and $v(t,x)$ is proportional to the horizontal velocity taken at the scaled height $\theta$ with $0\leq \theta \leq 1$ ($\theta=1$ at the free surface and $\theta=0$ at the bottom), and
\begin{equation*}
\displaystyle a = \left({\theta^2\over2} - {1\over6} \right)\nu,\quad  \displaystyle b = \left({\theta^2\over2} - {1\over6} \right)(1 - \nu) ,\quad
\displaystyle c = {(1-\theta^2)\over2}\mu - \tau, \quad  \displaystyle d = {(1-\theta^2)\over2}(1 - \mu)
\end{equation*}
with $\nu$ and $\mu$ arbitrary real numbers, and $\tau\geq0$ the surface tension. Throughout this paper we will just be dealing with the pure gravity waves, that is, $\tau=0$.

Preliminary work on well-posedness of \eqref{abcdsys} is detailed in \cite{BCS2} for various collections of $a, b, c$, and $d$. The authors proved that all the linearly well-posed systems are locally nonlinearly well-posed. Notice that an interesting feature of \eqref{abcdsys} is that its dispersive properties can be quite different according to the choice of $a, b, c$, and $d$. In fact the dispersive relation is
\begin{equation*}
\omega^2(\xi) = - |\xi|^2 {(1 -  a |\xi|^2)(1 -  c |\xi|^2) \over (1 +  b |\xi|^2)(1 +  d |\xi|^2)},
\end{equation*}
and hence the $(a, c)$ pair enhances dispersion while the $(b, d)$ pair weakens it. In this paper we will concentrate on the case of weak dispersion, that is, $a=c=0$ and $b=d={1\over6}$. After rescaling $b$ and $d$ to 1 the system becomes
\begin{equation}\label{bbm2}
\left\{
\begin{split}
 & \eta_t - \eta_{txx} + v_x + (\eta v)_x = 0, \\
 & v_t - v_{txx} + \eta_x + vv_x = 0, \\
\end{split}\right.
  \quad x \in \mathbb{R},  \; t \ge 0.
\end{equation}

One of our goals is to obtain a sharp well-posedness result of the above system. It has been shown in \cite{BCS2} through a contraction principle that system \eqref{bbm2} is well-posed in Sobolev spaces $H^s(\mathbb{R})\times H^s(\mathbb{R})$ for $s\geq0$. Using the idea of Bejenaru-Tao \cite{BeTao} we manage to prove, via explicitly constructed initial data, that the flow map fails to be continuous below $L^2(\mathbb{R})\times L^2(\mathbb{R})$. The idea stems from looking at the second order iterates in a Picard iteration about the chosen initial data, which leads to a {\em high-to-low frequency interaction}: a solution starts off initially with small energy and Fourier transform supported primarily at high frequencies, but quickly generates a large energy at low frequencies. This norm explosion indicates that the solution operator, if it exists at all, has a severe singularity near the zero solution.

Our ill-posedness threshold agrees with the one for the BBM equation \cite{BoTz}. However, system \eqref{bbm2} involves both the surface $\eta$ and velocity $v$ components and consequently generates many more interactions in the second iteration map than in the scalar BBM case. Our initial data is chosen to minimize the surface-velocity interactions but to maximize the frequency interactions of the linear propagators in the nonlinear terms. Similar techniques were used in proving a $C^3$ ill-posedness of the full gravity-capillary water wave problem \cite{CMSW}.

Note that the weak dispersive system \eqref{abcdsys} corresponding to $a=c=0$ does possess global-in-time solutions. For the particular value $\theta^2={7\over9}$, the system features exact solitary-wave solutions \cite{Ch1}. However for general $\theta\in[0,1]$ the existence of solitary waves is still open.

From modeling point of view, it is interesting to consider the following scaled asymptotic model of \eqref{bbm2}
\begin{equation}\label{scaledbbm2}
\left\{
\begin{split}
 & \eta_t + v_x + \varepsilon \LB (\eta v)_x - \eta_{txx} \RB  = 0, \\
 & v_t + \eta_x + \varepsilon \LB vv_x - v_{txx} \RB = 0. \\
\end{split}\right.
\end{equation}
A natural but important question about the asymptotic water wave models is the justification of their hydrodynamic relevance, which involves two issues, whether or not  (i) solutions to the full water wave equations which satisfy the modulational Ansatz exist for a sufficiently long time, and (ii) solutions to the asymptotic models approach the solutions to the full water waves equations within certain time.

It was proved in \cite{AL} that the existence time of the full water wave equations in the long wave regime ($\varepsilon = O(\delta)$) is  order $O(1/\varepsilon)$, and hence the large-time existence of solutions is provided. In regard to the second question, it was shown  in \cite{BonaColinLannes, Chazel} that {\it all} Boussinesq models give an approximation of order $O(\varepsilon^2 t)$ to the full water waves equations, provided the solutions exist. Therefore, the justification of  system \eqref{scaledbbm2} requires a large time well-posedness of the Cauchy problem.

For two dimensional surfaces, it was shown by Bona-Colin-Lannes \cite{BonaColinLannes} using a hyperbolic approach that the  existence time of solutions to system \eqref{abcdsys} is  order $O(1/\varepsilon)$  in $H^{{3\over2}+}(\mathbb{R}^2)$ when $a=c$ and the system is symmetric. In the strongly dispersive case when $a=c=1/6$, with $b=d=0,  \text{ and } \tau = 0$, it was proved by Linares-Pilod-Saut \cite{LPS} that the solutions to system \eqref{abcdsys} exist in a time scale of order $O(1/\sqrt{\varepsilon})$ in $H^{{3\over2}+}(\mathbb{R}^2)$. On the other hand, it was also proved in \cite{DMS} that when $b, d>0$ system \eqref{abcdsys} is well-posed in the Sobolev space $H^1(\mathbb{R}^2)$ on the time scale of order $O({\varepsilon}^{-{1\over2}+})$. The only success in reaching the optimal time scale for general $a, b, c$ and $d$, to our knowledge, is due to Ming-Saut-Zhang \cite{MSZ} for $a, c \leq0 $ and $  b, d > 0$, where a Nash-Moser approach is applied when the initial data are smooth (and with a loss of derivatives).

For one dimensional surfaces, for properly chosen initial data $ \eta_0 $ and $ v_0 $ with $ v_0 = \eta_0 + O(\varepsilon) \eta_0^2, $  it was proved in \cite{AABCW} that the solutions to system \eqref{bbm2} persists on the time scale of order $O(1/\varepsilon)$ with an extra high regularity assumption. One of our goals here is to obtain the well-posedness of system \eqref{scaledbbm2} ($a=c=0, b=d>0$) on time scales of order $O(1/\sqrt{\varepsilon})$ for general initial data in a much rougher space, achieving the rigorous justification of the model on the corresponding time scales. In general, we, however, are not able in our functional setting to reach the optimal time scales of order $O(1/\varepsilon)$.

Another aspect of this paper concerns the formation of singularities of system \eqref{bbm2}. Applying an energy-type argument we show that solutions blow up in finite time when $v_x$ is not bounded below. Together with the use of the Hamiltonian functional, we further prove that blow-up in finite time happens if and only if $\eta$ becomes unbounded below,  which improves the result of global existence obtained in \cite{AABCW} and also  seems to agree with some numerical predictions in \cite{Ch2}.

Since dissipation effects are ignored in the derivation of \eqref{abcdsys} and the overlying Euler system is Hamiltonian, it is expected that some special cases of \eqref{abcdsys} will also possess a Hamiltonian form. In particular, for our case system \eqref{bbm2}, the following functional
\begin{equation}\label{Hamiltonian}
\mathcal{H}(\eta,v) = {1\over2} \int_{\mathbb{R}} \LB \eta^2 + (1+\eta) v^2 \RB \ dx
\end{equation}
serves as a Hamiltonian. In addition, system \eqref{bbm2} has the following conserved quantities
\[
\int_\mathbb{R} \eta\ dx, \quad \int_\mathbb{R} v\ dx, \quad \int_\mathbb{R} \LC \eta v + \eta_x v_x \RC \ dx
\]
along with $\mathcal{H}(\eta, v)$ (\cite{BCS2}).

The plan of the paper is as follows. In Section 2, we introduce the notations and function spaces to work with. In Section 3, we establish the analytical well-posedness of system \eqref{bbm2} in Sobolev spaces $H^s(\mathbb{R})\times H^s(\mathbb{R})$ for $s\geq0$, and we further show that the time scale of existence is of order $\varepsilon^{-1/2}$. In Section 4, we prove that system \eqref{bbm2} is $C^0$ ill-posed below $L^2(\mathbb{R})\times L^2(\mathbb{R})$. In Section 5, we discuss the formation of singularity and provide a series of blow-up criteria. Finally in Appendix, we give the decay estimates of the linearized equation.
\vskip 0.2cm

\section{Resolution space}\label{sec_notation}

In this section we introduce a few notation and we define our functional
framework.

Let $ \vec u(t,x) = ( \eta(t,x), v(t,x)) $. 
Denote $ |\vec u(t)|_{p,q} \equiv |\eta(t)|_p + |v(t)|_q $ the norm in $ L^p \times L^q, $ for $ 1 \le p, \, q \le + \infty$, where  $|\cdot|_p $ refers to the $ L^p \equiv L^p(\mathbb{R}, dx) $ norm. Let $ \|\vec u(t) \|_{s} \equiv \|\eta(t)\|_{s} + \|v(t)\|_{s} $ for $ s \ge 0,$  denote the norm of Sobolev space  $ X^s  = H^s(\mathbb{R}) \times H^{s}(\mathbb{R}) $ where $ \| u\|_s $ is the norm of Sobolev space $ H^s(\mathbb{R}). $

For $u=u(t,x)\in \mathcal{S}'(\mathbb{R}^2)$, let $\hat{u}$ be its Fourier transform in space. We define the Japanese bracket $\left<x\right> = (1+|x|^2)^{1/2}$. The notation $A \sim B$ means that there exists a constant $c \geq 1$ such that ${1\over c} |A| \leq |B| \leq c|A|$. For any positive $A$ and $B$, the notation $A \lesssim B$ (resp. $A \gtrsim B$) means that there exists a positive constant $c$ such that $A \leq cB$ (resp.
$A \geq cB$).

Notice that one may rewrite \eqref{bbm2} as the following abstract form
\begin{equation}\label{absbbm2}
\vec u_t = A \vec u + N(\vec u)
\end{equation}
where
\begin{equation}\label{notation}
 A = - (1-\partial_x^2)^{-1} \partial_x \begin{pmatrix} 0 & 1 \\ 1 & 0 \end{pmatrix},  N(\vec u) = - (1-\partial_x^2)^{-1}\partial_x  \begin{pmatrix} \eta v  \\ {1 \over 2} v^2 \end{pmatrix}. 
\end{equation}
Then the linear semi-group $S(\cdot)$ associated with \eqref{bbm2} is given by
\begin{equation}\label{linop}
\LC\widehat{S(t)\vec\phi}\RC(\xi) = \begin{pmatrix}  \cos (t \xi \left < \xi \right >^{-2} ) & i  \sin (t \xi \left < \xi \right >^{-2} ) \\ i  \sin (t \xi \left < \xi \right >^{-2} ) & \cos (t \xi \left < \xi \right >^{-2} ) \end{pmatrix} \widehat {\vec \phi} (\xi).
\end{equation}

In this way, we will mainly work on the following integral formulation of \eqref{bbm2} with initial data $\vec{u}_0$
\begin{equation}\label{duhamel}
\begin{split}
{\vec u}(t) &= S(t)\vec u_0 + \int^t_0 S(t - t') N\LC\vec u(t') \RC\ dt'\\
&\equiv S(t) \vec u_0 + N_2 (\vec u, \vec u),
\end{split}
\end{equation}
where for $\vec u(t) = (u_1, u_2)(t)$ and $\vec v(t) = (v_1, v_2)(t)$
\begin{equation}\label{defnN2}
N_2({\vec u, \vec v}) = -{1\over2}\int^t_0 S(t-t')(1-\partial_x^2)^{-1}\partial_x \begin{pmatrix} u_1v_2 + u_2v_1\\ u_2v_2 \end{pmatrix} (t') \ dt'
\end{equation}
is a bilinear operator.

We can apply a Picard iteration (fixed point argument) to solve the above integral equation locally in time in the space $X^s_T = C([0, T], X^s)$ equipped with the usual norm
\begin{equation*}
\| \vec u \|_{X^s_T} = \sup_{[0, T]} \| \vec u(t, \cdot) \|_{s}.
\end{equation*}

Clearly, the free evolution operator $S(t)$ is unitary on $X^s$ for any $t\geq0$, that is
\begin{equation*}
\|S(t)\vec u_0\|_{s} = \|\vec u_0\|_{s},
\end{equation*}
and hence for any $T>0$,
\begin{equation}\label{linpart}
 \|S(t)\vec u_0\|_{X^s_T} = \|\vec u_0\|_{s}.
\end{equation}
Similarly,
\begin{equation}\label{nlpart}
\|N_2(\vec u, \vec u)\|_{X^s_T}\leq T\|S(t)N\LC\vec u(t) \RC\|_{X^s_T} = T\|N\LC\vec u(t) \RC\|_{X^s_T}.
\end{equation}
Therefore the Picard iteration can be used if we can establish a bilinear estimate of $\|N\LC\vec u(t) \RC\|_{X^s_T}$.

\section{Local well-posedness}\label{sec_local}

The goal of this section is to establish the local analytical well-posedness of system \eqref{bbm2}. As is explained above, we would first need a bilinear estimate for the nonlinear part in the integral formulation \eqref{duhamel}. This is achieved using the following lemma , which can be found in \cite{BoTz}.
\begin{lem}[\cite{BoTz}]\label{lem_bilnear}
Let $u,v\in H^s(\mathbb{R})$, $s\geq0$. Then
\begin{equation}\label{bilinearest}
\|(1 - \partial_x^2)^{-1}\partial_x(uv)\|_{H^s}  \leq C_s \|u\|_{H^s}\|v\|_{H^s},
\end{equation}
where $C_s>0$ is some constant depending only on $s$.
\end{lem}

Applying the above lemma to \eqref{nlpart} we obtain
\begin{equation}\label{bilinearest}
\|N_2(\vec u, \vec u)\|_{X^s_T}\leq C_sT\| \vec{u} \|^2_{X^s_T}.
\end{equation}

The following lemma will be used to estimate the maximal time of existence for solutions.
\begin{lem}\label{lem_reg}
Let $s>s'>1/2$, $s'\geq s-1$ and $u,v\in H^s(\mathbb{R})$. Then
\begin{equation}\label{biforreg}
\|(1 - \partial_x^2)^{-1}\partial_x(uv)\|_{H^s}  \leq C_s \|u\|_{H^{s'}}\|v\|_{H^{s}},
\end{equation}
where $C_s>0$ is some constant depending only on $s$.
\end{lem}
\begin{proof}
First it is easy to see that
\[
\|(1 - \partial_x^2)^{-1}\partial_x(uv)\|_{H^s} \leq \|uv\|_{H^{s-1}}.
\]
When $s\geq 1$, from \cite{KP} and the assumption that $s'>1/2, s'\geq s-1$ we know that the elements of $H^{s'}(\mathbb{R})$ are multipliers in $H^{s-1}(\mathbb{R})$, that is
\[
\|uv\|_{H^{s-1}} \leq C_s \|u\|_{H^{s'}}\|v\|_{H^{s-1}},
\]
and hence we have \eqref{biforreg}.

When $s<1$,
\[
\|uv\|_{H^{s-1}}\leq \|uv\|_{L^2}\leq \|u\|_{H^{s'}}\|v\|_{L^2}\leq \|u\|_{H^{s'}}\|v\|_{H^{s}},
\]
thus we also arrive at \eqref{biforreg}.
\end{proof}

\subsection{Bejenaru-Tao's abstract well-posedness theory}\label{subsec_BeTaowell}

With the help of \eqref{linpart} and \eqref{bilinearest}, the main analytic well-posedness theorem is given as below, where we follow the idea established by Bejenaru and Tao \cite{BeTao}. For a complete statement and proofs of the abstract result please refer to \cite{BeTao}. Here we just provide an argument for the estimate of the maximal existence time.

\begin{thm}\label{thm_wellpose} (Local well-posedness in $X^0=L^2(\mathbb{R})\times L^2(\mathbb{R})$)\ 
Fix $s\geq0$, for any $\vec{u}_0(x) = (\eta_0(x), v_0(x)) \in X^s$, there exist a $T=T(\vec{u}_0)>0$ and a unique solution $\vec{u}(t,x) = (\eta(t,x), v(t,x)) \in X^s_T$ to equation \eqref{duhamel}, and hence to \eqref{bbm2} with initial value $\vec{u}_0$. The maximal existence time $T_{s} $ for the solution has the property that
\begin{equation}\label{maxtime}
T_{s} \geq {1\over 4C_s\|\vec{u}_0\|_{s}}
\end{equation}
where $C_s>0$, and
\begin{equation}\label{bu}
T_{s} < \infty \Rightarrow \limsup_{t\to T^-_{s}} \|\vec{u}\|_s = \infty.
\end{equation}
Moreover, $T_s$ is independent of $s$ if $s>1/2$.

More specifically, we can construct the iterations $A_n:\ X^s\to X^s_T$ for $n=1, 2, \ldots$ by the recursive formulae
\begin{equation}\label{iteration}
\begin{split}
A_1(\vec{u}_0) & = S(t)\vec{u}_0\\
A_n(\vec{u}_0) & = \sum_{n_1, n_2\geq1, n_1+n_2=n} N_2(A_{n_1}(\vec{u}_0), A_{n_2}(\vec{u}_0))\quad \text{for\ } n>1,
\end{split}
\end{equation}
with the property that
\begin{equation*}
\LN A_n(f) \RN_{X^s_T} \leq K^n\LN f \RN^n_{X^s}, \quad \text{for some } K>0,
\end{equation*}
and we have the absolutely convergent (in $X^s_T$) power series expansion
\begin{equation}\label{iteratesoln}
\vec{u} = \sum^\infty_{n=1} A_n(\vec{u}_0).
\end{equation}
This implies the analyticity of the solution map $\vec{u}_0\mapsto \vec{u}$ from bounded sets of $X^s$ to $X^s_T$.
\end{thm}
\begin{proof}
We only provide an argument for the maximal time of existence here. For any initial data $\vec{u}_0\in X^s$, define a map
\begin{equation*}
F: D \equiv\{ \vec{u}\in X^s_T:\ \LN \vec{u} \RN_{X^s_T} < R = 2 \|\vec{u}_0\|_{X^s} \} \rightarrow X^s_T
\end{equation*}
as
\[
F(\vec{u}) = S(t) \vec u_0 + N_2 (\vec u, \vec u).
\]
Then from \eqref{linpart} and \eqref{bilinearest}, for any $\vec{u}, \vec{w}\in D$,
\begin{align*}
&\|F(\vec{u})\|_{X^s_T} \leq \|\vec{u}_0\|_{s} + C_s T \|\vec{u}\|^2_{X^s_T} < {R\over2} + C_sTR^2,\\
&\|F(\vec{u}) - F(\vec{w})\|_{X^s_T} \leq C_sT\|\vec{u} + \vec{w} \|_{X^s_T} \| \vec{u} - \vec{w} \|_{X^s_T} < 2RC_s \| \vec{u} - \vec{w} \|_{X^s_T}.
\end{align*}
Hence by choosing
\[
T = {1\over 2C_sR} = {1\over 4C_s\|\vec{u}_0\|_{s}}
\]
the map $F$ becomes a contraction from $D$ into itself. Therefore for each $s\geq0$, the maximal existence time $T_{s}$ satisfies
\[
T_{s} \geq {1\over 4C_s\|\vec{u}_0\|_{s}},
\]
and $T_{s}$ is finite implies that the $X^s$-norm of solution blows up to infinity as $t\to T_{s}$. Hence we have proved \eqref{maxtime} and \eqref{bu}.

It is obvious that $T_{s}$ is a non-increasing function of $s$. Let $s>1/2$. For any $1/2<s'<s$, suppose that $T_s< T_{s'}$. Therefore by uniqueness the maximal $X^s$ and $X^{s'}$ solutions coincide on $[0, T_s)$. Also from the definition of $T_s$ and $T_{s'}$, $\| \vec{u}(t) \|_{{s'}}$ remains bounded on $[0, T_s]$ while $\| \vec{u}(t) \|_{{s}}$ blows up as $t\to T_s$. From \eqref{duhamel} and Lemma \ref{lem_reg} we have
\begin{equation}\label{density}
\begin{split}
\|{\vec u}(t)\|_s & \leq \| \vec u_0 \|_s + \int^t_0 \| N\LC\vec u(t') \RC \|_s \ dt'\\
&\leq \| \vec u_0 \|_s + \int^t_0 \LN (1 - \partial_x^2)^{-1}\partial_x \begin{pmatrix} \eta v \\ {1\over2}v^2 \end{pmatrix}(t') \RN_{s} \ dt'\\
&\leq \| \vec u_0 \|_s + C_s \int^t_0 \|{\vec u}(t')\|_{s'} \|{\vec u}(t')\|_s \ dt'.
\end{split}
\end{equation}
Then by Gronwall's inequality we conclude that $\|{\vec u}(t)\|_s$ is bounded on $[0, T_s)$, which is a contradiction. Therefore, $T_s = T_{s'}$ and hence it does not depend on $s$ when $s>1/2$.
\end{proof}

\subsection{Hydrodynamic relevance.}

As is discussed in the Introduction, it is of interest from modeling point of view to consider system \eqref{scaledbbm2}.

The result of Theorem \ref{thm_wellpose} does not indicate the dependence of the temporal interval $[0, T]$ on the parameter $\varepsilon$. From \cite{AL}, the physically relevant temporal regime for \eqref{scaledbbm2} is from $O(1/\varepsilon)$ up to $O(1/\varepsilon^2)$. Hence we present here the mathematical analysis on that perspective.

In an abstract form, the above system can be rewritten as
\begin{equation*}
\vec{u}_t = A^{\varepsilon} \vec{u} + N^\varepsilon(\vec{u})
\end{equation*}
where
\begin{equation*}
A^\varepsilon = - (1 - \varepsilon\partial_x^2)^{-1} \partial_x \begin{pmatrix} 0 & 1 \\ 1 & 0 \end{pmatrix},  \quad N^\varepsilon(\vec u) = - \varepsilon (1 - \varepsilon\partial_x^2)^{-1} \partial_x  \begin{pmatrix} \eta v  \\ {1 \over 2} v^2 \end{pmatrix}. 
\end{equation*}
The associated linear semi-group $S^\varepsilon(\cdot)$ is
\begin{equation*}
\LC\widehat{S^\varepsilon(t)\vec\phi}\RC(\xi) = \begin{pmatrix}  \cos (t \xi (1+\varepsilon\xi^2)^{-1} ) & i  \sin (t \xi (1+\varepsilon\xi^2)^{-1} ) \\ i  \sin (t \xi (1+\varepsilon\xi^2)^{-1} ) & \cos (t \xi (1+\varepsilon\xi^2)^{-1} ) \end{pmatrix} \widehat {\vec \phi} (\xi)
\end{equation*}
which is still unitary on $X^s$.

The Duhamel formulation of \eqref{scaledbbm2} is
\begin{equation*}
{\vec u}(t) = S^\varepsilon(t)\vec u_0 + \int^t_0 S^\varepsilon(t - t') N^\varepsilon\LC\vec u(t') \RC\ dt'.
\end{equation*}

The bilinear estimate is
\begin{lem}\label{lem_scaled}
Let $u,v\in H^s$, $s\geq0$. Then
\begin{equation}\label{scaledbilinearest}
\|\varepsilon (1 - \varepsilon\partial_x^2)^{-1}\partial_x(uv)\|_{H^s}  \leq C_s\sqrt{\varepsilon}\|u\|_{H^s}\|v\|_{H^s},
\end{equation}
where $C_s>0$ depends only on $s$.
\end{lem}
\begin{proof}
Using duality and a polarization argument, one may write \eqref{scaledbilinearest} in the equivalent form
\begin{equation}\label{duality}
\LV \int_{\mathbb{R}^2} {\varepsilon \xi \left< \xi \right>^s \over (1 + \varepsilon\xi^2) \left< \xi_1 \right>^s\left< \xi - \xi_1 \right>^s} \hat{u}(\xi_1) \hat{v}(\xi - \xi_1) \overline{\hat{w}(\xi)}\ d\xi d\xi_1 \RV \leq C_s \sqrt{\varepsilon} \|u\|_{L^2}\|v\|_{L^2}\|w\|_{L^2}.
\end{equation}
For $s\geq0$, ${\left< \xi \right>^s / \left< \xi_1 \right>^s\left< \xi - \xi_1 \right>^s} \leq C_s$ and hence such term in \eqref{duality} can be ignored. Let
\[
\widehat{w}_1(\xi) = {\varepsilon \xi \over 1 + \varepsilon\xi^2} \overline{\hat{w}(\xi)}.
\]
Then the left hand side of \eqref{duality} is $\left< \hat{u}\ast \hat{v}, \bar{w}_1 \right>$. Moreover, since $u(x)$ and $u_1(x) = \hat{u}(-x)$ have the same $L^2$-norm, we have
\begin{align*}
\left< \hat{u}\ast \hat{v}, \bar{w}_1 \right> & = \left< u_1\ast w_1, \bar{\hat{v}} \right> \leq \|u_1\ast w_1\|_{L^2} \|v\|_{L^2}\\
& \leq \|u\|_{L^2} \|w_1\|_{L^1} \|v\|_{L^2}\\
& \leq \LN {\varepsilon \xi \over 1 + \varepsilon\xi^2} \RN_{L^2} \|u\|_{L^2}\|v\|_{L^2}\|w\|_{L^2} \lesssim \sqrt{\varepsilon} \|u\|_{L^2}\|v\|_{L^2}\|w\|_{L^2},
\end{align*}
and hence proves the lemma.
\end{proof}

From the above lemma we have for $s\geq0$ that
\begin{equation*}
\LN \int^t_0 S^\varepsilon(t - t') N^\varepsilon\LC\vec u(t') \RC\ dt' \RN_{X^s_T} \leq C_s \sqrt{\varepsilon} T \|\vec{u}\|^2_{X^s_T}.
\end{equation*}
Consequently we obtain the existence of solutions on the time scale of order $O(1/\sqrt{\varepsilon})$.
\begin{thm}\label{thm_asymp1}
Fix $s\geq0$, for any $\vec{u}_0(x) = (\eta_0(x), v_0(x)) \in X^s$, there exist a $T=T(\vec{u}_0)>0$ and a unique solution $\vec{u}(t,x) = (\eta(t,x), v(t,x)) \in X^s_T$ to \eqref{scaledbbm2} with initial data $\vec{u}_0$. The maximal existence time $T = T_s$ for the solution has the property that
\begin{equation}\label{approxmaxtime}
T_s \geq {1\over 4\sqrt{\varepsilon}C_s\|\vec{u}_0\|_{s}}
\end{equation}
where $C_s>0$ depends only on $s$. When $s>1/2$ the maximal time $T_s$ is independent of $s$. Moreover, the solution map $\vec{u}_0 \mapsto \vec{u}$ is analytic from bounded sets of $X^s$ to $X^s_T$.
\end{thm}

\vspace{.1in}

\section{Ill-posedness}\label{sec_ill}

The ill-posedness result we discuss here can be viewed as an application of a general result
proved in \cite{BeTao}. Roughly speaking, this general ill-posedness result requires the
two following ingredients:
\begin{enumerate}
\item The equation is analytically well-posed until a certain critical index $s_c$ and the corresponding solution-map is also analytic in the space of solutions.
\item Below this critical index, one iteration of the Picard scheme is not continuous. The discontinuity may be driven by high-to-low frequency interactions that blow up in frequencies of order smaller or equal to one.
\end{enumerate}

The first ingredient is given by Theorem \ref{thm_wellpose} whereas the second one can be derived using the abstract ill-posedness theory in \cite{BeTao} (c.f. Proposion 1 in \cite{BeTao}), which says that if the map $\vec{u}_0 \mapsto \sum^\infty_{n=1} A_n(\vec{u}_0)$ is continuous in a coarse topology, then each component $\vec{u}_0 \mapsto A_n(\vec{u}_0)$ of the series is also continuous in this coarse topology. Therefore, we may reduce the discontinuity of the flow map to disproving a bilinear estimate, namely, we will show that the second iteration $A_2(\vec{u}_0)$ is not continuous.

Our ill-posedness result is stated below.

\begin{thm}\label{thm_illpose} (Ill-posedness below $X^0$)\
For any $s<0$, let $R>0$ be arbitrary. Denote $B_R$ the ball
\begin{equation*}
B_R = \{ \vec{u}_0\in X^s:\ \|\vec{u}_0\|_{s} <R \}.
\end{equation*}
Let $T$ and $\vec{u}_0\mapsto \vec{u}$ be as in Theorem \ref{thm_wellpose}. Then the solution map $\vec{u}_0\mapsto \vec{u}$ is discontinuous from $B_R$ (with the $X^s$ topology) to $X^{0}_T$ (with the $X^{s'}_T$ topology) for any $s'\in \mathbb{R}$.
\end{thm}

\begin{rem}
Note that for the one-component analogue of \eqref{bbm2}, c.f. the BBM equation \cite{BBM}, it was shown by Bona-Tzvetkov \cite{BoTz} that the associated Cauchy problem is $C^2$ ill-posed below $L^2(\mathbb{R})$ by proving that the second Picard iteration fails to be continuous. In fact their problem fits in the framework of Bejenaru-Tao, and hence one can infer that the BBM equation is $C^0$ ill-posed below $L^2$. Our result here suggests the same threshold for the system \eqref{bbm2}, but the frequency interactions here are more delicate.
\end{rem}

\subsection{Second iteration $A_2(\vec{u}_0)$}

By definition \eqref{iteration}, we know
\begin{equation*}
A_2(\vec{u}_0) = N_2(S(t)\vec{u}_0, S(t)\vec{u}_0) = \int^t_0 S(t-t') N(S(t')\vec{u}_0)\ dt'.
\end{equation*}

Our goal is to construct initial data $\vec{u}_0$ that violates
\begin{equation}\label{continuitybound}
\| A_2(\vec{u}_0) \|_{X^{s'}_T} \lesssim \| \vec{u}_0 \|^2_{s} \quad \text{ for  } s<0.
\end{equation}

For simplification of notation, we denote
\begin{align*}
S(t) \begin{pmatrix}
\eta \\
v
\end{pmatrix}  =  \begin{pmatrix} {L}_1 \eta  +  {L}_2  v \\ {L}_2  \eta + {L}_1  v \end{pmatrix} ,
\end{align*}
where
\begin{equation}\label{phase}
\begin{split}
\widehat{L}_1(\xi, t)  &  = \cos(\lambda (\xi) t)  \\
\widehat{L}_2(\xi, t) & = i\sin(\lambda( \xi)t),
\end{split}
\end{equation}
with $\lambda(\xi) = \xi \left< \xi \right>^{-2}$. For an initial datum $\vec{u}_0 = \begin{pmatrix} \eta_0 \\ v_0 \end{pmatrix}$, let $\begin{pmatrix} \eta_1 \\ v_1 \end{pmatrix}(t) = S(t)\begin{pmatrix} \eta_0 \\ v_0 \end{pmatrix}$. Then
\begin{equation*}
A_2(\vec{u}_0) = \int_0^t S(t-t') (1-\partial_x^2)^{-1} \partial_x \begin{pmatrix} \eta_1 v_1 \\ {1\over2} v_1^2 \end{pmatrix}(t')\ dt' \equiv \int^t_0 \begin{pmatrix} Q_1 \\ Q_2 \end{pmatrix}\ dt'.
\end{equation*}
Taking the Fourier Transform of the above second order iterations yields the following set of terms
\begin{align*}
\widehat{Q}_1 & = {i\xi \over 1 + \xi^2} \LB \widehat{L}_1 (\xi,t-t') \widehat{\eta_1 v_1}(\xi, t') + \widehat{L}_2 (\xi,t-t') \widehat{{1\over2} v^2_1}(\xi, t') \RB\\
& = {i\xi \over 1 + \xi^2} \int_{\mathbb{R}} \widehat{L}_1 (\xi,t-t') \LC \widehat{L}_1 (\xi_1,t') \widehat{\eta}_0(\xi_1) + \widehat{L}_2(\xi_1, t') \widehat{v}_0(\xi_1) \RC \\
& \quad \quad \quad \quad \times \LC \widehat{L}_2(\xi - \xi_1, t')\widehat{\eta}_0(\xi-\xi_1) + \widehat{L}_1(\xi-\xi_1, t') \widehat{v}_0(\xi-\xi_1) \RC d\xi_1 + \\
& \quad \ \ {i\xi \over 2 (1 + \xi^2)} \int_{\mathbb{R}} \widehat{L}_2 (\xi,t-t') \LC \widehat{L}_2 (\xi_1,t') \widehat{\eta}_0(\xi_1) + \widehat{L}_1(\xi_1, t') \widehat{v}_0(\xi_1) \RC \\
& \quad \quad \quad \quad \times \LC \widehat{L}_2(\xi - \xi_1, t')\widehat{\eta}_0(\xi-\xi_1) + \widehat{L}_1(\xi-\xi_1, t') \widehat{v}_0(\xi-\xi_1) \RC d\xi_1,
\end{align*}
\begin{align*}
\widehat{Q}_2 & = {i\xi \over 1 + \xi^2} \LB \widehat{L}_2 (\xi,t-t') \widehat{\eta_1 v_1}(\xi, t') + \widehat{L}_1 (\xi,t-t') \widehat{{1\over2} v^2_1}(\xi, t') \RB \\
& = {i\xi \over 1 + \xi^2} \int_{\mathbb{R}} \widehat{L}_2 (\xi,t-t') \LC \widehat{L}_1 (\xi_1,t') \widehat{\eta}_0(\xi_1) + \widehat{L}_2(\xi_1, t') \widehat{v}_0(\xi_1) \RC \\
& \quad \quad \quad \quad \times \LC \widehat{L}_2(\xi - \xi_1, t')\widehat{\eta}_0(\xi-\xi_1) + \widehat{L}_1(\xi-\xi_1, t') \widehat{v}_0(\xi-\xi_1) \RC d\xi_1 + \\
& \quad \ \ {i\xi \over 2 (1 + \xi^2)} \int_{\mathbb{R}} \widehat{L}_1 (\xi,t-t') \LC \widehat{L}_2 (\xi_1,t') \widehat{\eta}_0(\xi_1) + \widehat{L}_1(\xi_1, t') \widehat{v}_0(\xi_1) \RC \\
& \quad \quad \quad \quad \times \LC \widehat{L}_2(\xi - \xi_1, t')\widehat{\eta}_0(\xi-\xi_1) + \widehat{L}_1(\xi-\xi_1, t') \widehat{v}_0(\xi-\xi_1) \RC d\xi_1.
\end{align*}
Of the above 16 terms in the two Duhamel operators $\widehat{Q}_1$ and $\widehat{Q}_2$, we look for a dominant lower bound term that violates an assumed upper bound in \eqref{continuitybound}.   This lower bound will arise from the following construction.

\subsection{Choice of data}

For a large $N$ we choose initial data $\vec{u}_0 = (\eta_0, v_0)$ so that
\begin{equation}\label{illposeddata}  \begin{split}
\widehat{\eta}_0 & = 0 , \\
\widehat{v}_0 &= N^{-s} {\chi}_{\{N-{1\over2}  \leq |\xi | \leq N+{1\over 2}\}},
\end{split} \end{equation}
where $\chi$ is the characteristic function. In this way we eliminate all interactions in $A_2(\vec{u}_0)$ involving $\eta_0$. Then our expressions for $\widehat{Q}_1$ and $\widehat{Q}_2$ reduce to  the following:
\begin{align*}
\widehat{Q}_1 & = {i\xi \over 1 + \xi^2}  \int_{\mathbb{R}} \LB \widehat{L}_1 (\xi,t-t') \widehat{L}_2(\xi_1, t')\widehat{L}_1(\xi-\xi_1, t') + {1\over2} \widehat{L}_2 (\xi,t-t')\widehat{L}_1(\xi_1, t')\widehat{L}_1(\xi-\xi_1, t')   \RB \\
& \quad \quad \quad \quad \quad \times \widehat{v}_0(\xi_1)\widehat{v}_0(\xi-\xi_1)\ d\xi_1,
\end{align*}
\begin{align*}
\widehat{Q}_2 & = {i\xi \over 1 + \xi^2}  \int_{\mathbb{R}} \LB \widehat{L}_2 (\xi,t-t') \widehat{L}_2(\xi_1, t')\widehat{L}_1(\xi-\xi_1, t') + {1\over2} \widehat{L}_1 (\xi,t-t')\widehat{L}_1(\xi_1, t')\widehat{L}_1(\xi-\xi_1, t')   \RB \\
& \quad \quad \quad \quad \quad \times \widehat{v}_0(\xi_1)\widehat{v}_0(\xi-\xi_1)\ d\xi_1.
\end{align*}

\begin{prop} \label{lowerboundprop}
If $N$ is sufficiently large, then on the support of $\widehat{v}_0(\xi_1)\widehat{v}_0(\xi-\xi_1)$ and for $0 \leq t' \leq t \leq  1$ we have the following bound:
\begin{equation*}
\widehat{L}_2 (\xi,t-t') \widehat{L}_2(\xi_1, t')\widehat{L}_1(\xi-\xi_1, t') + {1\over2} \widehat{L}_1 (\xi,t-t')\widehat{L}_1(\xi_1, t')\widehat{L}_1(\xi-\xi_1, t')  \geq {1\over32}.
\end{equation*}
\end{prop}
\begin{proof}
Since $N$ is large we know that on the support of $\widehat{v}_0(\xi_1)\widehat{v}_0(\xi-\xi_1)$
\begin{equation*}
{1\over 2N} \leq |\lambda(\xi_1)|, |\lambda(\xi-\xi_1)| \leq {1\over N}.
\end{equation*}
So when $0 \leq t' \leq t \leq  1$,
\begin{align*}
1 \geq  \widehat{L}_1(\xi_1, t'), \widehat{L}_1(\xi-\xi_1, t')  \geq {1\over2}.
\end{align*}
As for $\widehat{L}_2$, we know that
\begin{align*}
\LV \widehat{L}_2(\xi_1, t') \RV = \LV \sin( \lambda(\xi_1) t') \RV \leq \LV \sin\LC {t'\over N} \RC \RV\leq {t'\over N} \leq {1\over N}.
\end{align*}
Hence similarly
\begin{align*}
\LV \widehat{L}_2(\xi-\xi_1, t') \RV \leq {1\over N}.
\end{align*}
Because on the support of $\widehat{v}_0(\xi_1)\widehat{v}_0(\xi-\xi_1)$, $N-{1\over2} \leq |\xi_1|, |\xi - \xi_1| \leq N+{1\over2}$, we deduce that $|\xi|\leq1$ or $2N-1 \leq |\xi| \leq 2N+1$.

If $|\xi|\leq 1$, then
\[
\LV \lambda(\xi) \RV \leq |\lambda(1)| \leq {1\over2}.
\]
So
\begin{align*}
\widehat{L}_1(\xi, t - t') & = \cos \LC \lambda(\xi) (t - t') \RC \geq \cos\LC {1\over2} \RC \geq {1\over2},\\
\LV \widehat{L}_2(\xi, t - t') \RV&  = \LV \sin \LC \lambda(\xi) (t - t') \RC \RV \leq \sin\LC {1\over2} \RC \leq {1\over2}.
\end{align*}
Therefore we have
\begin{align*}
\LV \widehat{L}_2 (\xi,t-t') \widehat{L}_2(\xi_1, t')\widehat{L}_1(\xi-\xi_1, t') \RV & \leq {1\over 2} \cdot {1\over N} \cdot 1 = {1\over 2N}, \\
{1\over2} \widehat{L}_1 (\xi,t-t')\widehat{L}_1(\xi_1, t')\widehat{L}_1(\xi-\xi_1, t') & \geq {1\over 2} \cdot {1\over 2} \cdot {1\over 2} \cdot {1\over 2} = {1\over 16}.
\end{align*}
Hence by choosing $N$ sufficiently large, say, $N \geq 16$ we obtain the desired bound.

If $2N-1 \leq |\xi| \leq 2N+1$, then $\LV \lambda(\xi) \RV \leq {1\over N}$ and hence
\begin{align*}
\widehat{L}_1(\xi, t - t') \geq {1\over2}, \quad \LV \widehat{L}_2(\xi, t - t') \RV \leq {1\over N}.
\end{align*}
Therefore
\begin{align*}
\widehat{L}_2 (\xi,t-t') & \widehat{L}_2(\xi_1, t')\widehat{L}_1(\xi-\xi_1, t') + {1\over2} \widehat{L}_1 (\xi,t-t')\widehat{L}_1(\xi_1, t')\widehat{L}_1(\xi-\xi_1, t')\\
& \geq {1\over 16} - {1\over 2N^2} \geq {1\over 32}
\end{align*}
for $N\geq 4$.
\end{proof}
\begin{rem}
Here we only estimate the symbols in $\widehat{Q}_2$ because we are seeking a lower bound of $\| A_2(\vec{u}_0) \|_{X^{s'}_T}$. From the definition of the symbols $\widehat{L}_1$ and $\widehat{L}_2$ we see that a lower bound can be obtained only when there is a multiplier consisting of all $\widehat{L}_1$-symbols. This is why we assign $\widehat{\eta}_0=0$, since if $\widehat{v}_0=0$ then the corresponding $\widehat{Q}_1$ and $\widehat{Q}_2$ are
\begin{align*}
\widehat{Q}_1 & = {i\xi \over 1 + \xi^2}  \int_{\mathbb{R}} \LB \widehat{L}_1 (\xi,t-t') \widehat{L}_1(\xi_1, t')\widehat{L}_2(\xi-\xi_1, t') + {1\over2} \widehat{L}_2 (\xi,t-t')\widehat{L}_2(\xi_1, t')\widehat{L}_2(\xi-\xi_1, t')   \RB \\
& \quad \quad \quad \quad \quad \times \widehat{\eta}_0(\xi_1)\widehat{\eta}_0(\xi-\xi_1)\ d\xi_1,
\end{align*}
\begin{align*}
\widehat{Q}_2 & = {i\xi \over 1 + \xi^2}  \int_{\mathbb{R}} \LB \widehat{L}_2 (\xi,t-t') \widehat{L}_1(\xi_1, t')\widehat{L}_2(\xi-\xi_1, t') + {1\over2} \widehat{L}_1 (\xi,t-t')\widehat{L}_2(\xi_1, t')\widehat{L}_2(\xi-\xi_1, t')   \RB \\
& \quad \quad \quad \quad \quad \times \widehat{\eta}_0(\xi_1)\widehat{\eta}_0(\xi-\xi_1)\ d\xi_1,
\end{align*}
and all the multipliers contain the symbol $\widehat{L}_2$.
\end{rem}

\subsection{Ill-posedness}

Choosing initial data as in \eqref{illposeddata}, we have
\begin{equation}\label{initialnorm}
\|\vec{u}_0\|_{s} = \LN \begin{pmatrix} \eta_0 \\ v_0 \end{pmatrix} \RN_{s} = \LN v_0 \RN_{H^s} = \LN \LA \xi \RA^s \widehat{v}_0 \RN_{L^2} \sim 1.
\end{equation}

\begin{proof}[Proof of Theorem \ref{thm_illpose}]
Fix $s<0$ and $s'$. Suppose for contradiction that the corresponding solution map is continuous. Then from Proposition 1 in \cite{BeTao}, we have \eqref{continuitybound}.

From \eqref{maxtime} and \eqref{initialnorm} we may rescale $T$ to be 1. Consider a low-frequency region $I = \{ |\xi|\leq 1 \}$, and therefore $ \LA \xi \RA^{s'} \sim 1$. It then follows from Proposition \ref{lowerboundprop} that
\begin{align*}
\| A_2(\vec{u}_0) \|_{X^{s'}_T} & = \sup_{0\leq t \leq T} \LN \LA \xi \RA^{s'} \int^t_0 \begin{pmatrix} Q_1 \\ Q_2 \end{pmatrix} \ dt' \RN_{L_{\xi}^2\times L_{\xi}^2} \\
& =  \sup_{0\leq t \leq T} \LN \LA \xi \RA^{s'} \int^t_0 Q_1 dt' \RN_{L_{\xi}^2} + \sup_{0\leq t \leq T} \LN \LA \xi \RA^{s'} \int^t_0 Q_2 dt' \RN_{L_{\xi}^2} \\
& \geq  \sup_{0\leq t \leq T} \LN  \LA \xi \RA^{s'} \int^t_0 Q_2 \ dt' \RN_{L^2_{\xi}}\\
& \geq {1\over 32} \LN  \LA \xi \RA^{s'} {i\xi \over 1+\xi^2} \int_0^{1} \int_{\mathbb{R}} \widehat{v}_0(\xi_1)\widehat{v}_0(\xi-\xi_1)\ d\xi_1  \ dt' \RN_{L_{\xi}^2(I)}\\
& \gtrsim N^{-s} \LN \xi \RN_{L^2(I)} \gtrsim N^{-s},
\end{align*}
which violates \eqref{continuitybound} when $s<0$.
\end{proof}

\section{Blow-up criteria}\label{sec_bu}

In the previous two sections we establish the local well-posedness and ill-posedness of solutions to system \eqref{bbm2}. Attention is now given to the formation of singularities.

As is pointed out in \cite{BCS2}, \eqref{bbm2} possesses a Hamiltonian structure with the Hamiltonian functional
\begin{equation*}
\mathcal{H}(\vec{u}) = {1\over2} \int_\mathbb{R} [\eta^2 + (1+\eta) v^2 ]\ dx,
\end{equation*}
and hence for any local solution $\vec{u}\in X^s$ with $s>1/2$, $\mathcal{H}(\vec{u})$ is conserved.

Now we give our first blow-up criterion.
\begin{thm}\label{thm_blowup}
Let $s>3/2$. A solution $\vec{u} = (\eta, v)$ of \eqref{bbm2} with initial data $\vec{u}_0\in X^s$ blows up in finite time $T<\infty$ if and only if
\begin{equation}\label{bucond}
\liminf_{t\to T} \LB \inf_{x\in\mathbb{R}} v_x(t,x) \RB = -\infty.
\end{equation}
\end{thm}
\begin{proof}

We first prove the result for some $s\geq 2$. Multiplying \eqref{bbm2} by $\vec{u}$ and using integration by parts we obtain
\begin{equation}\label{ibp1}
\begin{split}
{d\over dt} \|\vec{u}\|^2_{1} &= {d\over dt} \int_{\mathbb{R}} \LC \eta^2 + \eta_x^2 + v^2 +v_x^2 \RC \ dx \\
& = -\int_{\mathbb{R}} \LB v_x\eta + (\eta v)_x\eta + \eta_xv + v^2v_x \RB\ dx\\
& = - \int_\mathbb{R} v_x \eta^2\ dx.
\end{split}
\end{equation}
Thus applying Gronwall's inequality we know that the $X^1$-norm of the solution does not blow-up in finite time if and only if $v_x(t, \cdot)$ is bounded from below on $[0, T)$.

Next we use \eqref{density} in Theorem \ref{thm_wellpose} to conclude that the $X^r$-norm of solution does not blow up for all $1\leq r\leq s$ if and only if $v_x(t, \cdot)$ is bounded from below on $[0, T)$. Therefore we proved the statement of the theorem for all $s\geq2$. Lastly, a density argument and the continuous dependence on initial data ensures the validity for all $s>3/2$.
\end{proof}
\begin{rem}
The result in Theorem \ref{thm_blowup} indicates the blow-up of the slope of velocity. But it is quite different from the phenomenon of ``wave-breaking" (solution remains bounded but the slope becomes infinity in finite time) which happens in a class of shallow water models \cite{CH1,Co1,Co2,Co3,CE,Mc,Wh}. For our system \eqref{bbm2}, the conservation laws are not strong enough to provide {\it a priori} bounds for the solutions. Though in Theorem \ref{thm_blowup} the regularity index is assumed to be above $3/2$, from Theorem \ref{thm_wellpose} we see that when $s>1/2$ (hence the functions are pointwise defined), the maximal time of existence is independent of the regularity index $s$. Hence this suggests that one may be able to look for a blow-up criterion in a lower regularity setting. In fact we can further refine our blow-up criterion as below, which asserts that it is indeed the pointwise blow-up of the solution itself rather than its slope.
\end{rem}

\begin{thm}\label{thm_busurf}
Let $s>1/2$. A solution $\vec{u} = (\eta, v)$ of \eqref{bbm2} with initial data $\vec{u}_0\in X^s$ blows up in finite time $T<\infty$ if and only if
\begin{equation}\label{bucondsurf}
\liminf_{t\to T} \LB \inf_{x\in\mathbb{R}} \eta(t,x) \RB = -\infty.
\end{equation}
\end{thm}
\begin{proof}

``$\Rightarrow$": Similarly, we first prove the result for $s\geq 1$. If not, suppose that $\eta(t,x) \geq -M$ on $0\leq t <T$ for some $M>0$. It follows from the conservation law of the Hamiltonian $\mathcal{H}$ that
\begin{align*}
\int_\mathbb{R} \eta^2 \ dx = \mathcal{H}(\vec{u}_0) - \int_\mathbb{R} (1+\eta) v^2 \ dx \leq \mathcal{H}(\vec{u}_0) + M  \int_\mathbb{R} v^2 \ dx.
\end{align*}

On the other hand, multiplying the second equation in \eqref{bbm2} by $2v$, integrating over $\mathbb{R}$, and using the above inequality we have
\begin{align*}
{d\over dt}\int_\mathbb{R} v^2 + v_x^2 \ dx & = 2 \int_\mathbb{R} v_x\eta \ dx \leq |v_x|_2^2 + |\eta|^2_2\\
& \leq \mathcal{H}(\vec{u}_0) + M |v|^2_2 + |v_x|_2^2\\
& \leq \mathcal{H}(\vec{u}_0) + (M + 1) \|v\|^2_{H^1}.
\end{align*}
Then using Gronwall's inequality we deduce
\begin{equation*}
\|v\|_{H^1} \leq C(\vec{u}_0, T).
\end{equation*}

Now multiplying the first equation in \eqref{bbm2} by $2\eta$, integrating over $\mathbb{R}$, and using the estimate for $\|v\|_{H^1}$ we have
\begin{align*}
{d\over dt}\int_\mathbb{R} \eta^2 + \eta_x^2 \ dx & = -2\int_\mathbb{R} \eta v_x\ dx + 2\int_\mathbb{R} \eta \eta_x v \ dx\\
& \leq 2 |v_x|_2 |\eta|_2 + |v|_\infty \LC |\eta|^2_2 + |\eta_x|^2_2 \RC\\
& \lesssim \|\eta\|_{H^1} + \|\eta\|^2_{H^1}.
\end{align*}
Hence
\begin{equation*}
\|\eta\|_{H^1} \leq C(\vec{u}_0, T).
\end{equation*}
In this way we see that the $X^1$-norm of the solution does not blow up. Then from \eqref{density} in Theorem \ref{thm_wellpose} we conclude that the $X^s$-norm of solution does not blow up, which is a contradiction.

Now for $1/2<s<1$, a density argument and the continuous dependence on initial data ensures the same result.

\vspace{.1in}

\noindent ``$\Leftarrow$":  This is can be proved by a use of the Sobolev embedding $H^s(\mathbb{R}) \hookrightarrow L^\infty(\mathbb{R})$ for $s>1/2$.
\end{proof}

\begin{rem}
(1) From the above two theorems it is clearly seen that when $s>3/2$, the blow-up of the slope of the velocity is equivalent to the blow-up of the surface profile.\\
(2) From Theorem \ref{thm_busurf} we also obtain a criterion for global existence of solutions: for $s>1/2$, if $\eta(t,x)$ is bounded from below on $t\in[0,T_{\max})$ , then $T_{\max} = \infty$. This is an improvement of the result in \cite{AABCW}, where the global existence of solutions requires that $\eta(t,x)> -1 + \alpha$ for some $\alpha>0$.
\end{rem}

Finally we give a blow-up characterization of the velocity profile.
\begin{cor}\label{cor_bu}
Let $s>1/2$. A solution $\vec{u} = (\eta, v)$ of \eqref{bbm2} with initial data $\vec{u}_0\in X^s$ blows up in finite time $T<\infty$ if and only if
\begin{equation}\label{bucondvelo}
\limsup_{t\to T} |v(t, \cdot)|_\infty = \infty.
\end{equation}
\end{cor}
\begin{proof}
``$\Rightarrow$": If not, suppose $|v|_\infty \leq M$ on $0\leq t <T$ for some $M>0$. From \eqref{ibp1}
\begin{equation*}
{d\over dt} \|\vec{u}\|^2_{1} = - \int_\mathbb{R} v_x \eta^2\ dx = 2 \int_\mathbb{R} v \eta \eta_x \ dx \leq 2 M \|\eta\|^2_{H^1}.
\end{equation*}
So by Gronwall's inequality, $\|\vec{u}\|_{1}$ is bounded. Hence $|\eta|_\infty$ remains bounded, which is a contradiction.

\vspace{.1in}

\noindent ``$\Leftarrow$": Again we can prove this by the Sobolev embedding $H^s(\mathbb{R}) \hookrightarrow L^\infty(\mathbb{R})$ for $s>1/2$.
\end{proof}

\begin{appendix}
\section{Linear estimates}\label{app_linest}

In this section, attention is now given to $ L^1 \to L^{\infty} $ estimate of solutions to the linear equation of \eqref{bbm2}.

\begin{thm}\label{thm_linest}
Consider the linearized equation of \eqref{bbm2}
\begin{equation}\label{linbbm2}
\vec u_t - A\vec u = 0
\end{equation}
with $ \vec u(0) = \vec u_0 = \begin{pmatrix} \eta_0 \\ v_0 \end{pmatrix}$.  If $ \vec u_0 \in X^1 \cap (L^1\times L^1)$, then $ \vec u (t) \in X^1 \cap (L^{\infty}\times L^{\infty}) $ and satisfies
\begin{equation}\label{linest}
| \vec u (t) |_{\infty,\infty} \lesssim (1 + t)^{- {1 \over 8}} ( \| \vec u_0 \|_1 + |\vec{u}_0|_{1,1} ), \qquad {\rm for }\;  t > 0.
\end{equation}
\end{thm}

To prove Theorem \ref{thm_linest}, we need the following two lemmas. The first one is the well-known Van der Corput lemma which we state without proof \cite{SoSt, St}.

\begin{lem}\label{lem_vander}
Let $ h $ be either convex or concave on $ [a, b] $ with $ - \infty \le a < b \le \infty. $ Then
\begin{equation*}
\left | \int_a^b e^{i h(\xi) } d\xi \right | \le 4 \{ \min_{[a, b]} |h'' | \}^{ - 1/2} \qquad {\rm for }\;   h'' \neq 0 \quad {\rm in }\,  [a, b].
\end{equation*}
\end{lem}

Applying the above Van der Corput lemma to the linearized equation \eqref{linbbm2}, we obtain
\begin{lem}\label{lem_vanderlinear}
Let $ n, t > 1 $ and $ 0 < \epsilon < 1. $  Then
\begin{equation}\label{suplinest}
\sup_{\alpha \in {\bf R}} \left | \int_{|\xi | \le n } e^{ i t h (\xi, \alpha ) } d \xi \right | \lesssim  \left ( \epsilon + t^{ - {1 \over 2}} \max \{ \epsilon^{- { 1 \over 2}}, n^{3 \over 2} \} \right )
\end{equation}
where $ \displaystyle h(\xi, \alpha ) = { \xi \over 1 + \xi^2 } + \alpha \xi $.
\end{lem}

\begin{proof} We use the stationary phase method \cite{Liu, SoSt} and Lemma \ref{lem_vander}. Since
$$
{\partial h \over \partial \xi } = { 1 - \xi^2 \over ( 1 + \xi^2)^2 } + \alpha \qquad {\rm and }
$$
$$
{\partial^2 h \over \partial \xi^2 } = { 2 \xi ( \xi^2 - 3) \over ( 1 + \xi^2 )^3},
$$
for $ \epsilon $ small and $ n $ large enough, it follows from Lemma 2.2 that
$$
\left | \int_{ \epsilon < |\xi | < \sqrt 3 - \epsilon } e^{i t h(\xi, \alpha ) } d \xi \right | \le 4 t^{ - { 1 \over 2}} \min \left \{ \left | { \partial^2 h \over \partial \xi^2} (\epsilon, \alpha) \right |^{-{1 \over 2 }}, \;  \left | { \partial^2 h \over \partial \xi^2} (\sqrt 3 - \epsilon, \alpha) \right |^{-{1 \over 2}}  \right \} \le c \epsilon^{ - { 1 \over 2}} t^{ - { 1\over 2}}
$$
and
\begin{align*}
\LV \int_{ \sqrt 3 + \epsilon < |\xi | < n } e^{i t h(\xi, \alpha ) } d \xi \RV & \le 4 t^{ - { 1 \over 2}} \min \left \{ \left | { \partial^2 h \over \partial \xi^2} (n,  \alpha) \right |^{-{ 1 \over 2}},   \left | { \partial^2 h \over \partial \xi^2} (\sqrt 3 + \epsilon, \alpha) \right |^{-{ 1 \over 2}}  \right \} \\ & \le c  t^{ - { 1 \over 2}} \max \{n^{3 \over 2}, \epsilon^{ - { 1 \over 2}} \}  .
\end{align*}
On the other hand,
$$
\left | \int_{|\xi | < \epsilon } e^{i t h(\xi, \alpha ) } d \xi \right |  \le c \epsilon
$$
and
$$
 \left | \int_{ \sqrt 3 - \epsilon < |\xi | < \sqrt 3 + \epsilon } e^{i t h(\xi, \alpha ) } d \xi \right | \le c \epsilon.
$$
Therefore, the above estimates imply \eqref{suplinest}.
\end{proof}

\vskip 0.2cm
\begin{proof}[Proof of Theorem \ref{thm_linest}.]  One can write the solution  $ \vec u = S(t) \vec u_0 $ of \eqref{linbbm2} as
\begin{equation}\label{linsoln}
\vec u (t) = \int^{\infty}_{-\infty}  e^{i x \xi} \begin{pmatrix}  \cos (t \xi \left < \xi \right >^{-2} ) & i  \sin (t \xi \left < \xi \right >^{-2} ) \\ i  \sin (t \xi \left < \xi \right >^{-2} ) & \cos (t \xi \left < \xi \right >^{-2} ) \end{pmatrix} \widehat {\vec u_0} (\xi) d \xi
\end{equation}
where $\widehat{\vec {u}_0} $ is the Fourier transform of $ \vec {u}_0 $.  It follows that
\begin{equation*}
|u(t)|_{\infty,\infty}  \le c \Sigma \left | \int^{\infty}_{-\infty} ( \widehat {\eta}_0 \pm  \widehat {v}_0 ) e^{it h( \xi,  \pm x \xi /t )} d \xi \right |,
\end{equation*}
where the sum $ \Sigma $ is over all two sign combinations. Using Lemma \ref{lem_vanderlinear}, we obtain
\begin{align*}
|\vec u(t) |_{\infty, \infty} & \lesssim \int_{|\xi| \ge n} \left ( | \widehat {\eta}_0| + | \widehat {v}_0 | \right ) d \xi +  \left | \int_{| \xi | < n } ( \widehat { \eta}_0  \pm \widehat {v}_0 ) e^{ i t h ( \xi, \pm { x \over t } )} d \xi \right | \\
& \lesssim \| \vec{u}_0 \|_1 \left ( \int_{|\xi| \ge n } ( 1 + \xi^2 )^{-1} d \xi \right )^{1 \over 2} \\
& \quad +  \int_{\mathbb{R}} \LC | \eta_0 (x) | + | v_0 (x) | \RC d x \left | \int_{|\xi | < n} e^{ it h (\xi, \pm {x \over t} )} d \xi \right | \\
& \lesssim n^{-{1 \over 2}}  \| \vec{u}_0 \|_1  +  \left ( \epsilon + t^{ -{1 \over 2}} \max \{ \epsilon^{- {1 \over 2}}, n^{3 \over 2} \} \right )  |\vec{u}_0 |_{1,1}.
\end{align*}
Therefore the estimate \eqref{linest} can be obtained by choosing $ \epsilon = t^{ - { 1 \over 3}} $ and $ n = t^{1 \over 4}. $ Since $ S(t) $ is a unitary group  in $ X^s $ for any $ s \ge 0$, we know $ \vec u(t) = S(t) \vec u_0 \in X^1.$
\end{proof}

\end{appendix}


\end{document}